\documentclass[a4paper,11pt]{article}

\usepackage[utf8]{inputenc}
\usepackage[T1]{fontenc}
\usepackage{amsmath,amssymb,amsthm}
\usepackage[]{authblk}
\usepackage{bbm}
\usepackage{color}
\usepackage{bbm}
\usepackage{verbatim}
\usepackage[colorlinks,linkcolor={black},citecolor={black},urlcolor={black}]{hyperref}

\usepackage[capitalise]{cleveref}

\usepackage{pgf,tikz}
\usetikzlibrary{arrows}

\usepackage{authblk}
\usepackage[utf8]{inputenc} 
\usepackage[T1]{fontenc}    
\usepackage{hyperref}       
\usepackage{url}           
\usepackage{booktabs}       
\usepackage{amsfonts}       
\usepackage{nicefrac}       
\usepackage{microtype}     
\usepackage{xcolor}       
\usepackage[maxbibnames=40, doi=false,isbn=false,url=false,eprint=false]{biblatex}
\usepackage[margin = 1.15in]{geometry}
\usepackage{graphicx}
\usepackage{mathtools}
\usepackage{mathrsfs}
\usepackage{dsfont}
\usepackage{caption}
\usepackage{subcaption}
\usepackage{wrapfig}
\usepackage{enumerate}
\usepackage[shortlabels]{enumitem}

\newtheorem{thm}{Theorem}
\crefname{thm}{Theorem}{Theorems}
\newtheorem*{thmm}{Theorem}
\crefname{thm}{Theorem}{Theorems}
\newtheorem*{qst}{Question}
\crefname{qst}{Question}{}
\newtheorem{cor}[thm]{Corollary}
\crefname{cor}{Corollary}{Corollaries}
\newtheorem{lem}[thm]{Lemma}
\crefname{lem}{Lemma}{Lemmas}
\newtheorem{prop}[thm]{Proposition}
\crefname{prop}{Proposition}{Propositions}

\crefname{conj}{Conjecture}{Conjectures}

\crefname{ques}{Question}{Questions}
\theoremstyle{definition}

\crefname{defn}{Definition}{Definitions}
\newtheorem{remark}[thm]{Remark}
\crefname{rem}{Remark}{Remarks}
\newtheorem{ex}[thm]{Example}
\crefname{ex}{Example}{Examples}

\crefname{obs}{Observation}{Observations}

\crefname{claim}{Claim}{Claims}

\crefname{ass}{Assumption}{Assumptions}

\numberwithin{thm}{section}

\renewcommand{\leq}{\leqslant}
\renewcommand{\geq}{\geqslant}
\renewcommand{\le}{\leqslant}
\renewcommand{\ge}{\geqslant}
\renewcommand{\to}{\rightarrow}

\DeclarePairedDelimiter{\norma}{\lvert}{\rvert} 
\DeclarePairedDelimiter{\abs}{\lvert}{\rvert}
\DeclarePairedDelimiter{\tond}{(}{)} 
\DeclarePairedDelimiter{\quadr}{[}{]}
\DeclarePairedDelimiter{\graf}{\{}{\}} 
\DeclarePairedDelimiter{\scal}{\langle}{\rangle}

\DeclareMathOperator*{\law}{law}

\newcommand{\numberset}{\mathbb}

\newcommand{\R}{\numberset{R}}

\newcommand{\hess}{\nabla^2}

\newcommand{\cova}[1]{\mathrm{Cov}_{#1}}

\newcommand{\vari}[1]{\mathrm{Var}_{#1}}

\newcommand{\dive}{\mathbf{\nabla} \cdot}

\newcommand{\expe}[1]{\mathbb{E}\quadr*{#1}}

\DeclareMathOperator{\lsi}{LSI}

\newcommand{\ddt}{\frac{\mathrm{d}}{\mathrm{d}t}}

\newcommand{\themaxt}{\theta_t^{\text{max}}}
\newcommand{\themint}{\theta_t^{\text{min}}}
\newcommand{\oulo}{\log Q_t\tond*{\frac{d\mu}{d\gamma}} }
\newcommand{\eee}{\mathrm{e}}

\begin{document}

\title{Heat flow, log-concavity, and Lipschitz transport maps}
\author{Giovanni Brigati\thanks{\textsf{giovanni.brigati@ist.ac.at}} }
\author{Francesco Pedrotti\thanks{\textsf{francesco.pedrotti@ist.ac.at}} }
\affil{Institute of Science and Technology Austria}

\maketitle

\begin{abstract}
In this paper we derive estimates for the Hessian of the logarithm (log-Hessian) for solutions to the heat equation. For initial data in the form of log-Lipschitz perturbation of strongly log-concave measures, the log-Hessian admits an explicit, uniform (in space) lower bound. This yields a new estimate for the Lipschitz constant of a transport map pushing forward the standard Gaussian to a measure in this class. Further connections are discussed with score-based diffusion models and improved Gaussian logarithmic Sobolev inequalities. 
Finally, we show that assuming only fast decay of the tails of the initial datum does not suffice to guarantee uniform log-Hessian upper bounds.
\end{abstract}



\section{Introduction}

Let $d \geq 1$. We say that a function $V\colon \R^d \to \R\cup\graf*{+\infty}$ is $\alpha$-convex, and that a probability density $\mu \in L^1_+(\R^d)$ is $\alpha$-log-concave, if, respectively, $x\to V(x) -\frac{\alpha}{2}\norma*{x}^2$ is convex, and $\mu (x) =  e^{-V(x)}$ for some $\alpha$-convex function such that $\int_{\R^d} {e}^{-V(x)} dx =1$. In case $\alpha =0$,  $\mu$~is a log-concave probability density; if $\alpha>0$, $\mu$ is strongly log-concave.
We also consider the \emph{heat flow} over $\R^d$:
\begin{equation}
    \label{heat}
    \begin{cases}
    \partial_t f = \frac{1}{2} \, \Delta f, \\
    \lim_{t\to 0} \, f(t,\cdot) = \mu.
    \end{cases}
\end{equation}
Taking $\mu = \delta_0$, the Dirac delta centered in zero, then the \emph{fundamental solution} to \eqref{heat} is $$f(t,x) = \gamma_t(x) := (2\pi t)^{-d/2} \, \mathrm{e}^{{-\norma {x}^2}/2t},$$
where $\gamma_t$ is the isotropic Gaussian density with zero mean and covariance matrix equal to $t I_d.$ 
Any other solution to \eqref{heat} is then given by $\mu * \gamma_t ,$ where $*$ is the symbol of convolution: $(g_1 * g_2) (x) = \int_{\R^d} g_1(x-y) \,g_2(y) \, dy.$ Denote by $(P_t)_t$ the corresponding heat semigroup, i.e.  
\begin{equation} \label{eq:heat-semigroup}
P_t \mu := \mu * \gamma_t, \quad t>0,
\end{equation}
which is induced by the flow of \eqref{heat}.
As solutions to \eqref{heat} are Gaussian convolutions of the initial datum $\mu$, it is expected that those would inherit some features from the Gaussian. There is a vast literature on the subject, which can be roughly classified into three types of results.  

\noindent\emph{(1).} Properties holding as soon as $t>0$. For example, for all $t>0$, $f(t,\cdot)$ is smooth  \cite{evans2022partial}.\\
\emph{(2)}. Asymptotic behaviour, in the limit $t \to \infty$, for which we refer to \cite{bakry1985diffusions,dolbeault2015best,vazquez2017asymptotic}.\\ 
\emph{(3).} Properties which are satisfied by $f(t,\cdot)$ for $t \geq T,$ after a finite time $T>0$. 

\subsection{Log-concavity in finite time}

Observing that the fundamental solution to \eqref{heat} is log-concave for all $t>0$, we pose the following, in the spirit of \emph{(3).}

\begin{qst} Given a probability measure $\mu$ on $\R^d$, does there exist a time $T>0$, such that the solution $f(t,x)$ to \eqref{heat} is log-concave for $t \geq T?$
\end{qst}
In general, we cannot expect \emph{instantaneous} creation of log-concavity, as suggested by the example $\mu = \frac{1}{2} ( \delta(1) + \delta(-1)) \in \mathcal{P}(\R)$, see \cite{brigati2023kinetic}. In addition, some hypotheses on the behaviour at infinity of $\mu$ shall be required, as suggested by \cite{herraiz1999asymptotic}. 
On the other hand, our question has a positive answer in two known cases.
\begin{itemize}
\item If $\mu$ is already log-concave, the solution to \eqref{heat} is log-concave at all times, see  \cite{saumard2014log,prekopa1973logarithmic,leindler1972certain,bra-lie-1976}. Then, by the semigroup property, if a solution to \eqref{heat} is log-concave at a time $T>0$, this property will be propagated to all $t \geq T.$ 
\item If $\mu$ is supported in $B(0,R)$, then $f(t,\cdot)$ is log-concave for all $t \geq R^2$, as pointed out first in \cite{lee2003geometrical}. More precisely, in \cite{bar-goz-mal-zit-2018} it is shown that for all $t>0$
\begin{equation}\label{eq:log-hess-bound-supp}
-\hess\log(\mu*\gamma_t) \succcurlyeq \frac{1}{t}\tond*{1-\frac{R^2}{t}} \,I_d.
\end{equation}
\end{itemize}

One aim of ours is to extend the class of measures for which creation of log-concavity in finite time holds, beyond the compactly supported case, motivated also by the series of papers \cite{ishige2021power,ishige2022new,ishige2023eventual,ishige2024hierarchy}, concerning various concavity property of solutions for the heat flow.   

An analogous question can be posed in the context of functional inequalities satisfied by the Gaussian distribution.  
Starting from the case of compactly supported measures,  previously analysed in \cite{zim-2013, wan-wan-2016,bar-goz-mal-zit-2018}, Chen, Chewi, and Niles-Weed prove in \cite{che-che-wee-2021} that if $\mu$ is subgaussian, i.e. for some $\epsilon, \mathcal{K}>0$

\begin{equation}\label{gautail}
\int_{\R^{d}} \mathrm{e}^{\epsilon \norma{x}^2} \, \mu(dx)  \leq \mathcal K,
\end{equation}
then the solution $\mu_t := f(t,\cdot) \, dx$ to \eqref{heat} satisfies a log-Sobolev inequality, for $t \geq T(\epsilon, \mathcal{K})$. Moreover, the subgaussianity assumption is also necessary. Indeed, if $\mu_T$ satisfies a log-Sobolev inequality for some $T>0$, then $\mu_T$ is also subgaussian \cite[Prop. 5.4.1]{bak-gen-led-2014}, which implies that $\mu$ is subgaussian in the first place. 
On the other hand, strongly log-concave measures do also satisfy a logarithmic Sobolev inequality, see \cite{bakry1985diffusions}. 
Then, one might wonder if \eqref{gautail} would be sufficient for a measure to become log-concave along the heat flow. 
The following theorem implies that this is not the case.
\begin{thm}\label{thm1}
For all non-decreasing function $\Psi\colon \R_{\ge 0}\to \R_{\ge 0}$, there exists an explicit probability measure on $\R$ such that
\begin{itemize}
    \item $\int_\R \mathrm{e}^{\Psi(x)} \mu(dx) < \infty$;
    \item for all $t> 0$, 
    $
        \inf_{x\in \R} \graf*{-\frac{d^2}{dx^2} \log \mu*\gamma_t} = -\infty.
    $

\end{itemize} 
\end{thm}
\begin{remark}
    Similar conclusions hold in arbitrary dimension, as it can be seen by considering the product probability measure $ \mu \times \delta_0 \times \ldots \times \delta_0$, with $\mu$ given by Theorem \ref{thm1}.
\end{remark}
Our result shows that the creation of log-concavity cannot be guaranteed by assuming only some control on the tails of the distributions $\mu$.
Therefore, we restrict our analysis to a perturbation regime, i.e.~we take measures $\mu$ which are close to being strongly log-concave, and we show that they become log-concave after a finite time along \eqref{heat}. 
More precisely, we prove the following
\begin{thm}\label{thm:log-lip-pert-log-conc-become-log-conc}
       Suppose that $\mu = \eee^{-(V+H)}\in L^1_+(\R^d)$, where $V\colon \R^d \to \R\cup\graf*{+\infty}$ is $\alpha$-convex and $H\colon\R^d\to\R$
        is $L$-Lipschitz for some $\alpha \in \R$, $L\ge0$.
        Then for every $t>0$ such that $\alpha t+1>0$ we have
        \begin{equation}\label{eq:log-hess-bound-conv-log-lip-pert}
            \frac{1}{t}\quadr*{1-\frac{1}{t}\tond*{\frac{L}{\alpha+\frac{1}{t}} + \sqrt{\frac{1}{\alpha+\frac{1}{t}}}}^2} I_d \preccurlyeq -\hess {\log(\mu*\gamma_t)} \preccurlyeq \frac{1}{t}I_d.
        \end{equation}
        In particular, for $\alpha>0$ and $t\ge \tond*{\frac{L}{\alpha} + \sqrt{\frac{1}{\alpha}}}^2$, we have that $\mu*\gamma_t$ is strongly log-concave.
\end{thm}
Equation \eqref{eq:log-hess-bound-conv-log-lip-pert} goes beyond the problem of log-concavity, yielding interesting consequences, as explained in the next subsections. 

\subsection{Application to Lipschitz transport maps}
\label{sec:lip-transp-map-intro}
In a seminal paper \cite{caf-2000}, Caffarelli showed that the Brenier map \cite{brenier1991polar} from optimal transport between the standard Gaussian $\gamma$ and an $\alpha$-log-concave probability measure $\mu$ is $(1/\sqrt{\alpha})$-Lipschitz. This result is useful because Lipschitz transport maps transfer functional inequalities (including isoperimetric, log-Sobolev and Poincar\'{e} inequalities) from a probability measure to another one, and it is typically much easier to prove these inequalities for the Gaussian measure in the first place.
For example, suppose that a probability measure $\mu$ satisfies the log-Sobolev inequality $\lsi(C)$ for some $C>0$, i.e. for all regular enough probability measures $\rho\ll \mu$

\begin{equation}\label{LSIc}\tag{LSI(C)} 
    \int \frac{d\rho}{d\mu} \log \frac{d\rho}{d\mu} \, d\mu  \le 2C \int \left|\nabla \sqrt{d\rho/d\mu} \right|^2 \, d\mu,
\end{equation}
where the two sides of the inequalities go under the name of \emph{relative entropy} and \emph{relative Fisher information}, respectively. Suppose, furthermore, that $T\colon \R^d\to\R^d$ is $L$-Lipschitz and consider the pushforward probability measure $\nu \coloneqq T\# \mu$. Then, $\nu$ satisfies $\lsi(L^2\cdot C)$.
Therefore, Caffarelli's result (together with the Gaussian LSI \cite{gross1975logarithmic}) immediately implies that strongly $\alpha$-log-concave probability densities satisfy $\lsi(1/\alpha)$, recovering the celebrated result by Bakry and \'Emery \cite{bakry1985diffusions}. Further details and many more applications of Lipschitz transport maps are discussed in \cite{mik-she-2023,cor_era-2002} and the references therein.

More recently, Kim and Milman \cite{kim-mil-2012} generalized Caffarelli's result by constructing another transport map, which is obtained by reverting an appropriate heat flow, and is referred to as the \emph{heat-flow map} (notation: $T^{\text{flow}}$). Other Lipschitz estimates for this transport map were then provided in \cite{mik-she-2023}, where the authors considered different types of assumptions on the target measure $\nu$ (namely, measures that satisfy a combination of boundedness and (semi-)log-concavity and some Gaussian convolutions). 

Several works dealt with the study of Lipschitz transport maps \cite{kla-put-2021,col-fat-2021, dai-gao-hua-jia-kan-liu-2023,mik-she-2021, she-2024,col-fig-jha-2017, car-fig-san-2024, che-poo-2023}; the recent paper \cite{fat-mik-she-2024} in particular considers an analogous class of target measure as in the present contribution. For comparison, we recall below its main result in the Euclidean setting.
\begin{thmm}[\cite{fat-mik-she-2024}, Theorem 1]
    Let $\mu=e^{-(V+H)},\, \nu =e^{-V}$ be  probability densities on $\R^d$ such that for all $x\in \R^d$ we have 
    $$
         \norma*{\nabla H}\le L,
        \qquad
         \hess V(x) \ge \alpha I_d,
        \qquad
         \abs*{\nabla^3 V(x)(w,w)}  \le K \qquad \text{ for all } w \in \mathbb{S}^{d-1},$$
    for some $\alpha>0, \, L, \,K\ge 0$.
    Then, there exists a transport map $T\colon \R^d\to \R^d$ such that $T\#\nu = \mu$ and $T$ is $\exp\tond*{\frac{5L^2}{\alpha}+\frac{5\sqrt{\pi}L}{\sqrt{\alpha}}+\frac{LK}{2\alpha^2}}$-Lipschitz.
\end{thmm}
\noindent Since Lipschitz transport maps can be composed, this result (combined with Caffarelli's theorem \cite{caf-2000}) implies in particular the existence of transport map $\tilde T$ such that $\tilde T \#\gamma = \mu$ and $\tilde{T}$ is Lipschitz with constant 
\begin{equation}\label{lip2}
  \frac{1}{\sqrt{\alpha}}\exp\tond*{\frac{5L^2}{\alpha}+\frac{5\sqrt{\pi}L}{\sqrt{\alpha}}+\frac{LK}{2\alpha^2}}.  
\end{equation}

On the other hand, we will prove in Section \ref{sec3} that our Theorem \ref{thm:log-lip-pert-log-conc-become-log-conc} implies new upper bounds on the Lipschitz norm for the heat-flow map from $\gamma$ to $\mu$.

\begin{thm}\label{thm3}
     Let $\mu = e^{-(V+H)}\in L^1_+(\R^d)$ be a probability density on $\R^d$ such that $V$ is $\alpha$-convex for $\alpha>0$ and $H$ is $L$-Lipschitz for $L\ge 0$. Then, there exists a map $T^\text{flow}\colon \R^d\to \R^d$ such that $T^\text{flow}\#\gamma = \mu$ and $T^\text{flow}$ is 
     $\frac{1}{\sqrt{\alpha}}\exp\tond*{\frac{L^2}{2\alpha}+2\frac{L}{\sqrt{\alpha}}}$-Lipschitz.
\end{thm}
\begin{remark}
    Consider the  case where $d=1$,  $V(x) = \frac12 {x}^2$ and $H(x) = L\abs*{x}+\log(Z)$ for a normalizing constant $Z$, so that the assumptions of Theorem \ref{thm3} are satisfied with $\alpha =1$. Then, it was observed in \cite{fat-mik-she-2024} that the Lipschitz norm of any  map $T$ such that $T\#\gamma = \mu$ is at least $\eee^{\frac{L^2}{2}}$. Hence, the dependence on $L^2$ in Theorem \ref{thm3} is sharp.
\end{remark}
The estimate for the Lipschitz constant of $T^\text{flow}$ in Theorem \ref{thm3} improves in particular on the value in  \eqref{lip2}, yielding the best available bound in this setting. Moreover, Theorem \ref{thm3} does not need any assumption on $\nabla^3 V$.

On the technical side, in Theorem \ref{thm3} we transport directly $\gamma$ to $\mu$ via the heat-flow map, and our proof only exploits elementary log-Hessian estimates for the heat semigroup, as in Theorem \ref{thm:log-lip-pert-log-conc-become-log-conc}. On the other hand, \cite{fat-mik-she-2024} employs a construction based on reverting the overdamped Langevin dynamics targeting the measure $\nu = \mathrm e^{-V}$: this requires estimates for the corresponding semigroup (cf. \cite[Proposition 2]{fat-mik-she-2024}), which is less explicit and needs more sophisticated arguments.
We remark that the results of \cite{fat-mik-she-2024} are of independent interest, due to the construction of a Lipschitz map transporting $\nu$ to $\mu$ therein, and the extension to some non-Euclidean spaces.

\subsection{Score-based diffusions models and the Gaussian LSI}

To further motivate our results, we briefly describe here two more applications of Theorem \ref{thm1} and \ref{thm:log-lip-pert-log-conc-become-log-conc}.
\paragraph{Score-based diffusion models} A similar construction as in Section \ref{sec:lip-transp-map-intro}, based on
reverting an ergodic diffusion process, has also recently found application in the machine learning community, within the framework of score-based diffusion models \cite{son-soh-kin-kum-erm-poo-2021,ho-jai-abb-2020}. Let $\mu$ be a probability measure, from which we want to generate random samples.
Consider the Ornstein--Uhlenbeck process (initialized at $\mu$)
\[
    X_0\sim \mu, \qquad dX_t = -X_tdt +\sqrt{2}dB_t,
\]
and denote by $Q_t$ the associate semigroup, i.e.
\begin{equation}\label{eq:ou-semigroup}
    Q_tf(x) = \int f\tond*{e^{-t}x+\sqrt{1-e^{-2t}}} \gamma(x) \,dx. 
\end{equation}
The key observation is that this process can be reverted, i.e. for $T_1>0$ the \emph{reverse SDE}
\begin{equation}
\label{eq:reverse-process}
    Y_0 \sim \law\tond*{X_{T_1}}, \qquad dY_t = -Y_t \,dt + 2\nabla  \oulo (Y_t) dt  +\sqrt{2}dB_t
\end{equation}
is such that $Y_{T_1}\sim \mu$, see \cite{and-1982, cat-con-gen-leo-2023, son-soh-kin-kum-erm-poo-2021}. Therefore, one can simulate the process $(Y_t)_t$ until time $T_1$ to  sample from $\mu$. A common assumption in theoretical works aimed at analysing this method is some control on the  Lipschitz constant of $\nabla \oulo$  \cite{che-che-lee-li-lu-sal-2023,che-che-lli-li-sal-zha-2023,che-lee-jia-2022} or on the one-sided one \cite{kwo-fan-lee-2022, ped-maa-mon-2023}. These assumptions are indeed useful to control the discretization errors when employing a numerical scheme to simulate the process or some sort of ``contractivity'' along the reverse dynamics.
On the one hand, Theorem \ref{thm:log-lip-pert-log-conc-become-log-conc} enlarges the class of distributions $\mu$ for which these assumptions can be justified, by implying bounds on the Hessian $\hess \oulo$ (cf. Corollary \ref{cor:ou-log-hess}), beyond the setting where the initial distribution $\mu$ has bounded support. 

On the other hand, Theorem \ref{thm1} shows that, for some distributions $\mu$, such assumptions can be too restrictive. Thus, complementary analysis is needed, as done  in
\cite{con-dur-sil-2023,ben-deb-dou-del-2023, che-lee-jia-2022}.
\paragraph{Improvements in the Gaussian LSI} 

The standard Gaussian measure $\gamma$ satisfies $\mathrm{LSI}(1)$. Henceforth, let $\nu\ll \gamma$ be a probability measure, and set $u^2 \coloneqq \frac{d \nu}{d \gamma} \in L^1(\gamma)$: then
\begin{equation}
    \label{LSIg}\tag{$\gamma$-LSI}
    \int_{\R^d} |\nabla u|^2 \, d\gamma - \frac{1}{2} \, \int_{\R^d} u^2 \, \log u^2 \, d\gamma \, \geq 0, \quad \text{ if } u \in \mathrm{H}^1(d\gamma).
\end{equation}
The Gaussian logarithmic Sobolev inequality was written first in \cite{gross1975logarithmic}, although it can be deduced from \cite{shannon1948mathematical}. 
The related literature is wide: see \cite{brigati2023stability} for a recent review, and \cite{rezakhanlou2008entropy,carlen1991superadditivity} for accurate historical comments. The constant $C=1$ is optimal, with extremal 
probability measures belonging to $\mathcal{M} \coloneqq \{\nu_{a,b} \coloneqq \mathrm{e}^{a+\scal{b, x}} \, \gamma$, \, $a \in \R, \, b \in \R^d$\}, according to \cite{carlen1991superadditivity}. Then, one may investigate whether the constant in the Gaussian LSI can be improved on a subclass of measures $\nu$, under orthogonality constraints. Contributions in this direction appear in \cite{fathi2016quantitative,brigati2023stability}, and they are closely related to \emph{stability inequalities}, for which the reader may refer to \cite{dolbeault2022sharp,dolbeault2024short,brigati2023logarithmic,brigati2023gaussian,ind-kim-2021}, and references quoted therein.   

In \cite[Theorem 1]{brigati2023stability}, an improved Gaussian LSI is shown: for all $\epsilon, \mathcal{K}>0$, there exists a constant $\eta(\epsilon,\mathcal{K})>0$,  such that for all probability measures $\nu =u^2 \, \gamma$ satisfying $\int_{\R^d} x \, \nu =0$, and  \eqref{gautail}, we have 
\begin{equation}
    \label{stab}
  \int_{\R^d} |\nabla u|^2 \, d\gamma - \frac{1}{2} \, \int_{\R^d} u^2 \, \log u^2 \, d\gamma \, \geq \eta \int_{\R^d} |\nabla u|^2 \, d\gamma, \qquad \text{if } u \in \mathrm{H}^1(d\gamma).
\end{equation}  
The core of the proof for \eqref{stab} is showing that -- after a finite time $T>0$ -- the solution $f(t,\cdot) = \nu * \gamma_t$ to \eqref{heat}, starting at $\nu$, satisfies a Poincar\'e inequality:  
$$\int_{\R^d} \varphi(x)^2 \, f(t,x) \, dx - \left( \int_{\R^d} \varphi(x) \, f(t,x) \, dx \right)^2  \leq C_P \, \int_{\R^d} |\nabla \varphi(x)|^2 \, f(t,x) \, dx,$$
for all functions $ \varphi \in \mathrm{H}^1(f(t,\cdot) \, dx),$ and some constant $C_P >0$. Condition \eqref{gautail} guarantees such a Poincar\'e inequality in finite time, see \cite{che-che-wee-2021}.
Then, \cite[Theorem 1]{fathi2016quantitative} applies, and an improved inequality like \eqref{stab} holds for $u(T,\cdot) = \sqrt{f(T,\cdot)/\gamma}$, after a finite time. The proof is completed by \emph{integrating backwards in time} via \cite[Lemma 2]{brigati2023stability}.

\begin{itemize}
   \item We notice first that, if $\nu = \mathrm{e}^{-(V+H)}$ is a log-Lipschitz perturbation of a strongly log-concave measure, then the ideas of \cite[Theorem 1]{brigati2023stability} apply, since \eqref{gautail} holds true.
   Alternatively, one could estimate the Poincar\'e constant of $f(t,\cdot) = \nu * \gamma_t$, for any $t\geq 0$, either via the Lipschitz transport map of Section \ref{sec3}, or by a perturbation argument \cite{cattiaux2022functional}, and apply \cite[Theorem 1]{fathi2016quantitative}. 
   Finally, one can optimise the resulting constant in \eqref{stab} over the parameter $t \geq 0$.   
   \item The scheme of proof for \cite[Theorem 1]{brigati2023stability} can be adapted to measures $\nu = u^2 \, d\gamma$ which become $\alpha$-log-concave in finite time along \eqref{heat}, for $\alpha>0$. In this case, the Poincar\'e inequality is given by the Bakry-\'Emery method \cite{bakry1985diffusions}.
   \item The same can be done for measures $\nu$ which become just log-concave ($\alpha =0$) in finite time along \eqref{heat}, provided an a priori bound on the second-order moment $\int_{\mathbb R^d} |x|^2\, \nu$, see \cite{bob-1999} and the discussion of \cite[Section 2]{brigati2023stability}.
\end{itemize}

\subsection{Structure of the paper}
The proof of Theorem \ref{thm:log-lip-pert-log-conc-become-log-conc} is given in Section \ref{sec2}, followed by Subsection \ref{subsec21}, where sufficient conditions in order to apply Theorem \ref{thm:log-lip-pert-log-conc-become-log-conc} are discussed. In Section \ref{sec3}, we detail our main application to the existence of Lipschitz transport maps, with the proof of Theorem \ref{thm3}. Finally, in Section \ref{sec4}, we prove the negative result for the creation of log-concavity in finite time, namely Theorem \ref{thm1}.

\subsection*{Acknowledgement} { 
The authors thank Professors Jean Dolbeault, Jan Maas, and Nikita Simonov for many useful comments, and Professors Kazuhiro Ishige, Asuka Takatsu, and Yair Shenfeld for inspiring interactions.
\\
This research was funded in part by the Austrian Science Fund (FWF) project \href{https://doi.org/10.55776/F65}{10.55776/F65} and by the European Union’s Horizon 2020 research and innovation programme under the Marie Sklodowska-Curie grant agreement No 101034413. 
}
\section{Log-Lipschitz perturbations of log-concave measures: proof of Theorem \ref{thm:log-lip-pert-log-conc-become-log-conc} }\label{sec2}
    Let $\mu$ be a probability measure on $\R^d$. For $t>0$ and $z\in \R^d$, define the probability measure $\mu_{z,t}$ by
    \begin{equation}
        \mu_{z,t}  \propto \exp \tond*{\frac{z\cdot x}{t} - \frac{\norma{x}^2}{2t}}\mu(x) \propto \gamma_{z, t}(x)\mu(x),
    \end{equation}
    where $\gamma_{z,t}$ is the Gaussian density with mean $z$ and covariance matrix $tI_d$.
    We will make frequent use of the following well-known probabilistic characterization of the Hessian of $\log(\mu*\gamma_t)$, cf. \cite{bar-goz-mal-zit-2018, kla-put-2021}:
    \begin{equation}\label{eq:prob-rep-log-hess}
        -\hess \log(\mu*\gamma_t)(z)  = \frac{1}{t}\tond*{I_d-\frac{\cova{\mu_{z,t}}}{t}}.  
    \end{equation}
   Consequently, bounds on $\hess \log(\mu*\gamma_t)$ are given by bounds on covariance matrices. For this purpose, we provide the following lemma, which gives an upper bound for the covariance matrix of a probability measure $\mu$ in terms of the covariance of another probability measure $\nu$ and of the Wasserstein distance between the two.
\begin{lem}\label{lem:vari-w2-close-measure}
        Let $\mu,\nu$ be probability measures on $\R^d$. For any unit vector $w \in \mathbb{S}^{d-1}$
        \begin{equation}
            \scal{w,\cova{\mu}\,  w} \le \tond*{W_2(\mu,\nu) + \sqrt{\scal{w,\cova{\nu}\,  w}}}^2.
        \end{equation}
    \end{lem}
    \begin{proof}
        Let $(X,Y)$ be an optimal coupling for $W_2(\mu,\nu)$. Fix a unit vector $w\in \R^d$ and let $X_w \coloneqq \scal{w,X}$ and $Y_w \coloneqq \scal{w,Y}$. 
        We have that
        \begin{align*}
            \scal{w,\cova{\mu}\,  w} &= 
            \expe{\tond*{X_w-\expe{X_w}}^2} \\
            &\le \expe{\tond*{X_w-\expe{Y_w}}^2} = \expe{\tond*{X_w-Y_w+Y_w-\expe{Y_w}}^2}
            \\
            & \le \tond*{\sqrt{\expe{\tond*{X_w-Y_w}^2}}+\sqrt{\expe{\tond*{Y_w-\expe{Y_w}}^2}}}^2 \qquad \text{ (by Cauchy--Schwarz)}
            \\
            & \le \tond*{W_2(\mu,\nu)+\sqrt{\expe{\tond*{Y_w-\expe{Y_w}}^2}}}^2  = \tond*{W_2(\mu,\nu) + \sqrt{\scal{w,\cova{\mu}\,  w}}}^2.
        \end{align*}
    \end{proof}
    \begin{proof}[Proof of Theorem \ref{thm:log-lip-pert-log-conc-become-log-conc}]
        The upper bound in \eqref{eq:log-hess-bound-conv-log-lip-pert} is well known, and holds for arbitrary probability measures $\mu$ (cf., for example, \cite[Lemma 1.3]{eld-lee-2018}); alternatively, it follows from \eqref{eq:prob-rep-log-hess} and the fact that covariance matrices are positive semidefinite.
        Let us then turn to the first inequality. Fix $t>0$ and $z\in \R^d$. 
        Define the probability density $\nu_{z,t}\in L^1_+(\R^d)$ by
        $
            \nu_{z,t} \propto e^{-V}\gamma_{z,t}. 
        $
        Notice that $\nu_{z,t}$ is $(\alpha+\frac{1}{t})$-log-concave: therefore, $\cova{\nu_{z,t}} \preccurlyeq\frac{1}{\alpha+\frac{1}{t}}I_d$ by the Brascamp--Lieb inequality \cite{bra-lie-1976} (cf. also \cite[Lemma 5]{eld-leh-2014}).
        Moreover we have $\mu_{z,t} \propto e^{-H}\nu_{z,t}$: it follows from \cite[Corollary 2.4]{khu-maa-ped-2024} that
        \[
            W_2\tond*{\mu_{z,t},\nu_{z,t}} \le W_\infty\tond*{\mu_{z,t},\nu_{z,t}} \le \frac{L}{\alpha+\frac{1}{t}}.
        \]
        We are now in position to apply Lemma \ref{lem:vari-w2-close-measure}: for any unit vector $v\in \R^d$ we have 
        \begin{align*}
            \scal{v,\cova{\mu_{z,t}}\,  v} &\le \tond*{W_2\tond*{\mu_{z,t},\nu_{z,t}} + \sqrt{\scal{v,\cova{\nu_{z,t}}\,  v}}}^2
            \\
            &\le \tond*{\frac{L}{\alpha+\frac{1}{t}} + \sqrt{\frac{1}{\alpha+\frac{1}{t}}}}^2.
        \end{align*}
        This shows that $\cova{\mu_{z,t}} \preccurlyeq \tond*{\frac{L}{\alpha+\frac{1}{t}} + \sqrt{\frac{1}{\alpha+\frac{1}{t}}}}^2 I_d$, and the  conclusion follows from \eqref{eq:prob-rep-log-hess}.
    \end{proof}

  \begin{remark}
       In the proof of Theorem \ref{thm:log-lip-pert-log-conc-become-log-conc}, we estimated from above  $W_2\tond*{\mu_{z,t},\nu_{z,t}}$ with the $L^\infty$-Wasserstein distance $W_\infty\tond*{\mu_{z,t},\nu_{z,t}}$. 
       Alternatively, we could have achieved the same conclusion as follows, using that $\nu_{z,t}$ satisfies $\lsi\tond*{\frac{t}{\alpha t+1}}$.  First, a transport-entropy inequality \cite{ott-vil-2000} allows to estimate $W_2\tond*{\mu_{z,t},\nu_{z,t}}$ in terms of the relative entropy of $\mu_{z,t}$ with respect to $\nu_{z,t}$; then, the relative entropy is bounded from above by the relative Fisher information using the logarithmic Sobolev inequality of $\nu_{z,t}$; finally, the relative Fisher information is easily estimated using that $\mu_{z,t}\propto e^{-H}\nu_{z,t}$ and $H$ is $L$-Lipschitz.
       \end{remark}

    \subsection{Sufficient conditions}\label{subsec21}
    By Theorem \ref{thm:log-lip-pert-log-conc-become-log-conc}, log-Lipschitz perturbations of strongly log-concave measures
    become log-concave in finite time along \eqref{heat}; by Theorem \ref{thm3}, they are the pushforward of the Gaussian measure via a Lipschtzt transport map.
    The purpose of this subsection is to give sufficient conditions for a measure $\mu$ 
    to be a log-Lipschitz perturbation of a strongly log-concave measure.
    \begin{ex}\label{ex:bound-supp-conv}
        Suppose that $\mu$ is a probability measure supported on the Euclidean ball $B(0,R)$ for some radius $R>0$. Then, proceeding as in \cite{bar-goz-mal-zit-2018}, for any $s>0$ we can write 
        \[
            \mu*\gamma_s = e^{-H}\gamma_s
        \]
        where $H\colon \R^d\to \R^d$ is $\frac{R}{s}$-Lipschitz. 
        By $\frac{1}{s}$-log-concavity of $\gamma_s$, Theorem \ref{thm:log-lip-pert-log-conc-become-log-conc} applied to  $ \mu*\gamma_{s}$ and $t>0$ yields that
        \[
             \frac{1}{t}\quadr*{1-\frac{1}{t}\tond*{\frac{R/s}{1/s+\frac{1}{t}} + \sqrt{\frac{1}{1/s+\frac{1}{t}}}}^2} I_d \preccurlyeq -\hess {\log(\mu*\gamma_{t+s})} \preccurlyeq \frac{1}{t+s} I_d.
        \]
      By letting $s\to 0$, we recover the classical estimate 
       \[
            \frac{1}{t}\tond*{1-\frac{R^2}{t}}I_d \preccurlyeq -\hess \log\tond*{\mu*\gamma_t}  \preccurlyeq \frac{1}{t}I_d,
       \]
       cf. \cite[Sec. 2.1]{bar-goz-mal-zit-2018}. In this sense, we can say that the class of densities considered in Theorem \ref{thm:log-lip-pert-log-conc-become-log-conc} contains both log-concave ones (taking  $L=0$) and the ones with bounded support.
    \end{ex}
    Consider now a probability density $\mu=\eee^{-U} \in L^1_+(\R^d)$ for some $U\in C^2(\R^d)$. The following result asserts that, if we have a uniform positive lower bound for the Hessian of $U$ outside some Euclidean ball, then we can rewrite $\mu$ as a log-Lipschitz perturbation of a strongly log-concave measure.
    \begin{lem}\label{lem:pos-hess-outside-ball}
     Let
     $U\in C^2(\R^d)$ be such that for some $\alpha,\beta, R \ge0$ it holds that
    	\[
    	\begin{cases}
    		\hess U(x) \succcurlyeq \alpha I_d & \text{ if } \norma{x} \ge R,
    		\\
    		\hess U(x) \succcurlyeq -\beta I_d & \text{ if } \norma{x} < R.
    	\end{cases}
    	\]
    	Then there exists $ V, H\in C^1( \R^d)$ such that $U =  V + H$, $ V$ is $\alpha$-convex and $H$ is $2(\alpha+\beta)R$-Lipschitz.
    \end{lem}
    \begin{proof}
         Let $ H\colon \R^d \to \R$ be defined by 
    	\[
    	- H(x) = 
    		\begin{cases}
    			(\alpha+\beta)\norma*{x}^2 & \text{ if } \norma{x}\le R,
    			\\
    			2(\alpha+\beta)R\norma{x}-
       {2}(\alpha+\beta)R^2 & \text{ if } \norma{x} \ge R,
    		\end{cases}
    	\]
    	and set $V(x) = U(x)- H(x)$. 
    	Then we have that $U =  V + H$, $ V \in C^1(\R^d)$ is $\alpha$-convex   and $\norma*{\nabla H} \le 2(\alpha +\beta )R$, as desired.
    \end{proof}

    The above lemma can be useful to study linear combinations of strongly log-concave densities, via the following 
\begin{prop}\label{prop5}
    Consider a measure $\mu = \sum_{i=1}^N \alpha_i \, \mathrm{e}^{-U_i}$   for some $N>0$, weights $\alpha_i> 0$ and potentials $U_i\in C^2(\R^d)$ such that $\eee^{-U_i}\in L^1_+(\R^d)$. 
    Assume $\nabla^2 U_i \succcurlyeq K  I_d$ for all $i$ and some $K>0$. 
    Then 
    \begin{align}\label{eq:bound-log-hess-mixture}
              - \nabla^2  \log \, \mu & \succcurlyeq 
       K I_d - \frac{\sum_{i>j} \alpha_i\alpha_j \eee^{-U_i-U_j}  (\nabla U_i - \nabla U_j)^{\otimes 2} }{\mu^2}  \\
       \label{eq:crude-bound-log-hess-mixture}
      & \succcurlyeq K I_d - \sum_{i>j} \frac{(\nabla U_i - \nabla U_j)^{\otimes 2}}{ \left(2 + \frac{\alpha_i}{\alpha_j}\mathrm{e}^{U_j-U_i} + \frac{\alpha_j}{\alpha_i}\mathrm{e}^{U_i-U_j}\right)}.
    \end{align}
   \end{prop}

\begin{proof}

Notice that
$$- \nabla^2  \log \, \mu = \mu^{-2} (\nabla \mu \otimes \nabla \mu - \mu \nabla^2 \mu).$$ Set $\mu_i :=\alpha_i \mathrm{e}^{-U_i}$ so that $\mu = \sum_{i=1}^N \mu_i$.
By construction     
$$\nabla \mu_i = - \nabla U_i \, \mu_i, \qquad \nabla^2 \mu_i = (- \nabla^2 U_i + \nabla U_i \otimes \nabla U_i) \mu_i, \quad \forall i=1,\ldots,N. $$
Then,     
        \begin{align*}
      - \nabla^2  \log \, \mu &= 
      \frac{\left( \sum_{i=1}^N \nabla U_i \, \mu_i \right)^{\otimes 2} - \left( \sum_{i=1}^N \mu_i \right) \, \left( \sum_{i=1}^N (- \nabla^2 U_i + \nabla U_i \otimes \nabla U_i) \mu_i \right) }{\mu^2}  \\
      &= \frac{\mu \sum_{i=1}^N \nabla^2 U_i \, \mu_i - \sum_{i,j=1}^N \mu_i \, \mu_j (\nabla U_i \otimes \nabla U_j - \nabla U_j \otimes \nabla U_j) }{\mu^2}  \\
      & \succcurlyeq K \, I_d- \frac{\sum_{i>j} \mu_i \mu_j (\nabla U_i - \nabla U_j)^{\otimes 2} }{\mu^2},
    \end{align*}       
    which shows \eqref{eq:bound-log-hess-mixture}.
    The crude estimate  
    \[
    \mu^2 = \sum_{l,m=1}^N  \mu_l \, \mu_m \geq 2 \mu_i \, \mu_j + \mu_i^2 + \, \mu_j^2 \quad \text{ for } i\neq j \]
    then gives \eqref{eq:crude-bound-log-hess-mixture}.
\end{proof}			
From the above proposition, it is clear that when the right-hand-side of \eqref{eq:bound-log-hess-mixture} is uniformly positive definite outside a Euclidean ball, then by Lemma \ref{lem:pos-hess-outside-ball} we can recast $\mu$ as a log-Lipschitz perturbation of a strongly log-concave measure. Therefore, the assumptions of Theorem \ref{thm:log-lip-pert-log-conc-become-log-conc} are satisfied, and $\mu*\gamma_t$  becomes strongly log-concave in finite time along the heat flow \eqref{heat}.  We illustrate this in the following example, where $\mu$ is a finite mixture of Gaussians in dimension $1$.

\begin{ex} 
    Let $\mu$ be a linear combination of one-dimensional Gaussians, i.e. $\mu = \sum_{i=1}^N \alpha_i \, \mathrm{e}^{-U_i}$   for some $N\ge 2$, weights $\alpha_i> 0$ and potentials $U_i$ of the form
    \[
       U_i(x) = \frac{(x-m_i)^2}{\sigma_i^2} 
    \]
    for some $m_i\in \R, \sigma_i^2>0$. 
    Without loss of generality we can assume that $U_i \neq U_j$ for $i\neq j$. 
    By Proposition \ref{prop5}, we have that 
    \begin{align*}
        -\frac{d^2}{dx^2} \log \mu \succcurlyeq  \frac{1}{\max_i \sigma_i^2} - \sum_{i>j} \frac{( U_i' -  U_j')^{2}}{ \left(2 + \frac{\alpha_i}{\alpha_j}\mathrm{e}^{U_j-U_i} + \frac{\alpha_j}{\alpha_i}\mathrm{e}^{U_i-U_j}\right)}.
    \end{align*}
    It is then not difficult to see that the argument of the sum in the right-hand-side converges to $0$ as $\abs*{x}\to \infty$. By the previous discussion, it follows that the assumptions of Theorem \ref{thm:log-lip-pert-log-conc-become-log-conc} are satisfied for some $L, \alpha >0$: hence, a finite linear combination of Gaussian densities on $\R$ becomes strongly log-concave in finite time along the heat flow.
\end{ex}

\section{Lipschitz transport maps: proof of Theorem \ref{thm3}}\label{sec3}

\paragraph{Construction of the heat-flow map.}
Let $\mu\in L^1_+(\R^d)$ be a probability density on $\R^d$.
Assume, furthermore, that $\mu$ has finite second-order moment. We begin by sketching the construction of the heat-flow map, and refer the reader to \cite{kim-mil-2012, mik-she-2023} for details.
The idea  is to interpolate between $\mu$ and $\gamma$ along the Ornstein--Uhlenbeck flow
\begin{equation}\label{eq:OU}
    X_0\sim \mu, \qquad  dX_t = -X_t dt+\sqrt{2}dB_t.
\end{equation}
Let us denote  by $Q_t$ the associated transition semigroup \eqref{eq:ou-semigroup} and by $\mu_t$ the law of $X_t$.
Then, $\mu_t$ satisfies the Fokker--Planck equation
\[
    \partial \mu_t -\dive\quadr*{\mu_t\nabla \oulo} = 0.
\]
Correspondingly, we can consider  the flow maps $(S_t)_{t\ge 0}$ obtained by solving 
\begin{align*}
    S_0(x) = x, \quad \ddt S_t(x) = -\nabla \oulo
\end{align*}
for all $x\in \R^d$.  Under some regularity assumptions (cf. \cite{kim-mil-2012,mik-she-2023,ott-vil-2000,vil-2003}), this defines a flow of diffeomorphisms such that $S_t\#\mu = \mu_t$; conversely, $T_t\coloneqq S_t^{-1}$ is such that $T_t\#\mu_t = \mu$. The heat-flow map is then heuristically defined by $T^{\text{flow}} = \lim_{t\to \infty} T_t$ and is such that $T^{\text{flow}}\#\gamma = \mu$.
To make things rigorous, we recall/adapt the following result from \cite{mik-she-2023}.
\begin{lem}\label{lem:tec-exis-lip-map}
    Suppose that $\mu\in L^1_+(\R^d)$ is a probability density with finite second-order moment. Suppose, furthermore, that for all $t>0$ there exist $\themaxt,\themint\in\R$ such that 
    \begin{align}
        \themint I_d \preccurlyeq \hess \oulo \preccurlyeq \themaxt I_d
    \end{align}
    and for all $s>1$
    \[
        \sup_{\frac{1}{s} <t<s} \max\graf*{\abs{\themint},\abs{\themaxt}} <\infty.
    \]
    Then, provided that $L\coloneqq \limsup_{t\to\infty}\int_{\frac{1}{t}}^t \themaxt dt<\infty$, there exists a  map $T\colon \R^d\to\R^d$ such that $T\#\gamma =\mu$ and $T$ is $e^L$-Lipschitz.
\end{lem}
\begin{proof}
    Notice first of all that $\mu_t$ is a smooth density for every $t>0$. Fix $s>0$: by the assumptions in the Lemma and by \cite[Lemma 2 and 3]{mik-she-2023} there exists a map $T_s$ which is $\exp\tond*{\int_{\frac{1}{s}}^s\themaxt \,dt}$-Lipschitz and such that $T_s\#\mu_s = \mu_{\frac{1}{s}}$. Since $\mu_s\to\gamma$ and $\mu_{\frac{1}{s}}\to \mu$ in $W_2$-distance (hence weakly) as $s\to \infty$, the  conclusion follows from \cite[Lemma 1]{mik-she-2023}.
\end{proof}

\paragraph{New estimates.}
In view of Lemma \ref{lem:tec-exis-lip-map}, the goal is to provide estimates on $\hess \oulo$, for some classes of probability measures $\mu$ on $\R^d$.
The Ornstein--Uhlenbeck semigroup $Q_t$ is related to the heat semigroup $P_t$ in \eqref{eq:heat-semigroup} by the  identity $Q_t f(x) = P_{1-e^{-2t}}f(e^{-t}x)$
for  $f\in L^1(\gamma)$. Combining this with
Theorem \ref{thm:log-lip-pert-log-conc-become-log-conc} yields the following
\begin{cor}[Corollary of Thm. \ref{thm:log-lip-pert-log-conc-become-log-conc}]\label{cor:ou-log-hess}
    Let $\mu = e^{-V-H}\in L^1_+(\R^d)$ be a probability density on $\R^d$ such that $V$ is $\alpha$-convex and $H$ is $L$-Lipschitz, for some $\alpha,\in \R,  L\ge 0$. 
    Then for every $0<t$ such that $\alpha t+1>0$ we have
\begin{equation}
\begin{split}
    \label{eq:ou-estimates}
    -\frac{1}{{e}^{2t}-1} \, I_d
    \preccurlyeq&
    \hess \oulo  
    \\
     \preccurlyeq &
     \tond*{\frac{1-\alpha}{\alpha\tond*{e^{2t}-1}+1} + \frac{e^{2t}L^2}{(\alpha(e^{2t}-1)+1)^2} +  \frac{2Le^{2t}}{\sqrt{\tond*{e^{2t}-1}} \, (\alpha\tond*{e^{2t}-1}+1)^{3/2}} } \, I_d.
    \end{split}
    \end{equation}
\end{cor}

\begin{proof}[Proof of Theorem \ref{thm3}]
We integrate the upper bound in \eqref{eq:ou-estimates}. An elementary computation using the change of variable $\tau = e^{2t}-1$  shows that

\begin{align*}
&\int_0^\infty \tond*{\frac{1-\alpha}{\alpha\tond*{e^{2t}-1}+1} + \frac{e^{2t}L^2}{(\alpha(e^{2t}-1)+1)^2} +  \frac{2Le^{2t}}{\sqrt{\tond*{e^{2t}-1}} \, (\alpha\tond*{e^{2t}-1}+1)^{3/2}} } \,dt
\\
= \,&  \int_0^\infty  \tond*{  \frac{1-\alpha}{\tau\alpha+1} + L^2 \frac{\tau+1}{(\tau\alpha+1)^2} +2L \, \frac{\tau+1}{\sqrt{\tau} \, (\tau\alpha+1)^{3/2}} } \frac{1}{2(\tau+1)}d\tau
\\
= \,& -\frac{1}{2} \log(\alpha) +\frac{L^2}{2\alpha} + 2\frac{L}{\sqrt{\alpha}}.
\end{align*}
The desired conclusion then follows from Lemma \ref{lem:tec-exis-lip-map}.
\end{proof}

\section{The negative result: proof of Theorem \ref{thm1}}\label{sec4}

Before proving the actual theorem, we give some heuristics behind the proof. 
The leading idea is the following. 
If one considers \eqref{heat} with $\mu = \delta_0$, then the solution is immediately log-concave for $t>0.$
However, this behaviour is not stable.

\begin{prop}\label{prop1}
   Fix $x_0 \in \R$. Let $\mu = \frac{\alpha}{\alpha+\beta}\delta_0 + \frac{\beta}{\alpha+\beta}\delta_{x_0}$, for some $\alpha,\beta>0$. Then, $\mu*\gamma_t$ is log-concave (if and) only if $t\ge \frac{1}{4} x_0^2$.
\end{prop} 
\begin{proof}
    We prove only the \emph{only if} part, since the other implication follows directly from \eqref{eq:log-hess-bound-supp}.  
    It is not difficult to see that with $x_0,t,\alpha, \beta>0$ fixed, there exists $\bar z\in \R$ for which 
    \[
         \alpha e^{-\frac{\bar z^2}{2t}} = \beta e^{-\frac{(\bar z-x_0)^2}{2t}}.
    \]
    Then, using \eqref{eq:prob-rep-log-hess}, we have that
    \[
         \frac{d^2}{dx^2} \tond*{-\log \mu*\gamma_t} (\bar z) =  \frac{1}{t }\tond*{1 - \frac{x_0^2}{4t}},
    \]
    which is negative if $t<x_0^2/4$.
\end{proof}

From equation \eqref{eq:log-hess-bound-supp} we see
that a compactly-supported distribution becomes log-concave along \eqref{heat} after a time $T=O(R^2)$. Proposition \ref{prop1} gives a simple account of this time scale being correct. 
In addition, we see that the time needed for the measure $\mu$ of Proposition \ref{prop1}  to become log-concave along \eqref{heat} does not depend on the mass of the perturbation $\delta_{x_0}$. Exploiting these observations allows us to create mixtures of Dirac deltas with arbitrarily thin tails, which never become log-concave along \eqref{heat}.

\begin{proof}[Proof of Theorem \ref{thm1}]   
    For $i\ge 0$, set $x_i = \frac{i(i+1)}{2} \ge 0$. 
    Define the probability measure $\mu$ on $\R$ by
    \[
        \mu \propto \sum_{i=0}^\infty \frac{1}{(i+1)^2}e^{-\Psi(x_i)}\delta_{x_i}
    \]
    and let $X\sim \mu$. It is immediate to check that $\expe{e^{\Psi(X)}}<\infty$.
    Let us now fix $t\ge 0$.
    Recall from \eqref{eq:prob-rep-log-hess} that 
    \[
        -\frac{d^2}{dx^2}\log \mu*\gamma_t(z) =  \frac{1}{t}\tond*{1-\frac{\vari{\mu_{z,t}}}{t}},
    \]
    where
    \[
        \mu_{z,t}(x) \propto e^{\frac{zx}{t}-\frac{x^2}{2t}}\mu(x) \propto \sum_{i=0}^\infty \frac{1}{(i+1)^2}e^{-\Psi(x_i)+\frac{zx_i}{t}-\frac{x_i^2}{2t}}\delta_{x_i}.
    \]
    Therefore, it suffices to prove that, for every $M>0$, there exists $z$ such that $\vari{\mu_{z,t}} \ge M^2$.
    To this end, fix $M$ and choose $j\ge\sqrt{2}M$ so that 
    \[
        \abs*{x_{j}-x_{j-1}}^2 = j^2 \ge 2M^2.
    \]
    To conclude, it suffices to show that there exists $z\in \R$ such that
    \begin{equation}\label{eq:equal-weight}
        \mu_{z,t}\tond*{[0,x_{j-1}]} = \frac{1}{2} = \mu_{z,t}\tond*{[x_{j}, +\infty]}.
    \end{equation}
    Indeed, the above implies that $\vari{\mu_{z,t}}\ge M^2$. 
    Notice now that \eqref{eq:equal-weight}
    is equivalent to finding a solution to the equation $F(z)=0$, where 
    \begin{align}\label{eq:def-F}
        F(z) = \sum_{i=0}^{j-1} \frac{1}{(i+1)^2}e^{-\Psi(x_i)+\frac{zx_i}{t}-\frac{x_i^2}{2t}} - \sum_{i=j}^\infty  \frac{1}{(i+1)^2}e^{-\Psi(x_i)+\frac{zx_i}{t}-\frac{x_i^2}{2t}}.  
    \end{align}
    It is straightforward to check that $F(0)\ge 0$, e.g. using that $1>\sum_{i=1}^\infty \frac{1}{(i+1)^2}$ and that $\Psi$ is non-decreasing.
    Moreover, $F$ is continuous, since for any compact interval $[a,b]\subset\R$, the series in \eqref{eq:def-F} converges uniformly in $C([a,b])$.
    To conclude, we show now that $\lim_{z\to \infty} F(z) = -\infty$. To this end, notice that 
    \begin{align*}
        F(z)  & \le  je^{-\Psi(0)+\frac{zx_{j-1}}{t}} - \frac{1}{(j+1)^2}e^{-\Psi(x_j)-\frac{x_j^2}{2t}+\frac{zx_j}{t}}
        \\
        & =  e^{\frac{zx_{j-1}}{t}} \tond*{je^{-\Psi(0)} - \frac{1}{(j+1)^2}e^{-\Psi(x_j)-\frac{x_j^2}{2t}}e^{\frac{zj}{t}}} , 
    \end{align*}
    which yields the desired conclusion since $j>0$. 
\end{proof}

\printbibliography

\end{document}